\documentclass[11pt, oneside]{article}   	
\usepackage{geometry}                		
\usepackage{graphicx}				
																	
\usepackage{amssymb}
\usepackage{amsmath}
\usepackage{amsthm}
\usepackage{stmaryrd}
\usepackage{tikz-cd}
\usepackage{extarrows}
\usepackage[mathscr]{euscript}
\usepackage{mathrsfs}
\usepackage[OT2,T1]{fontenc}
\DeclareSymbolFont{cyrletters}{OT2}{wncyr}{m}{n}
\DeclareMathSymbol{\Sha}{\mathalpha}{cyrletters}{"58}
\newcommand*{\A}{\ensuremath{\mathbf{A}}}

\usepackage{calligra}
\usepackage{mathrsfs}
\DeclareMathOperator{\Hom}{\mathscr{H}\text{\kern -3pt {\calligra\large om}}\,}
\DeclareMathOperator{\calExt}{\mathscr{E}\text{\kern -3pt {\calligra\large xt}}\,}
\DeclareMathOperator{\Ext}{Ext}
\DeclareMathOperator{\Pic}{Pic}
\DeclareMathOperator{\Jac}{Jac}

\newcommand*{\calO}{\mathcal{O}} 
\newcommand*{\Z}{\ensuremath{\mathbf{Z}}}
\newcommand*{\Q}{\ensuremath{\mathbf{Q}}}

\newcommand{\Ga}{{\mathbf{G}}_{\rm{a}}}
\DeclareMathOperator{\R}{R}
\newcommand{\calHom}{\mathscr{H}\mathit{om}}
\newcommand*{\address}{Einstein Institute of Mathematics, The Hebrew University of Jerusalem, Edmond J. Safra Campus, 91904, Jerusalem, Israel}
\newcommand*{\email}{zevrosengarten@gmail.com}

\usepackage{rotating}

\usepackage{bm}

\numberwithin{equation}{section}

\newtheorem{theorem}{Theorem}[section]
\newtheorem{lemma}[theorem]{Lemma}
\newtheorem{definition}[theorem]{Definition}
\newtheorem{proposition}[theorem]{Proposition}
\newtheorem{corollary}[theorem]{Corollary}

\theoremstyle{remark}
  \newtheorem{remark}[theorem]{Remark}
  
  \theoremstyle{remark}
  \newtheorem{example}[theorem]{Example}

\usepackage[OT2,T1]{fontenc}

\tikzset{commutative diagrams/.cd,
mysymbol/.style={start anchor=center,end anchor=center,draw=none}
}

\newcommand{\Gm}{{\mathbf{G}}_{\rm{m}}}

\usepackage{stmaryrd}
\usepackage{hyperref}

\title{TRANSLATION-INVARIANT LINE BUNDLES ON LINEAR ALGEBRAIC GROUPS}
\author{Zev Rosengarten \thanks{While completing this work, the author was supported by an ARCS Scholar Award, a Ric Weiland Graduate Fellowship, and a Zuckerman Postdoctoral Scholarship.}}

\date{}
\begin{document}
\maketitle

\begin{abstract}
We study the Picard groups of connected linear algebraic groups, and especially the subgroup of translation-invariant line bundles. We prove that this subgroup is finite over every global function field. We also utilize our study of these groups in order to construct various examples of pathological behavior for the cohomology of commutative linear algebraic groups over local and global function fields.
\end{abstract}

\tableofcontents{}

\section{Introduction}

Let $G$ be a linear algebraic group over a field $k$, i.e., a smooth affine $k$-group scheme. Sansuc \cite{sansuc} showed that when $G$ is connected and $k$ is perfect, the group $\Pic(G)$ of line bundles on $G$ up to isomorphism is finite. Over imperfect fields, however, this finiteness fails in general. For a smooth connected group scheme $G$ over an arbitrary field $k$, let $m, \pi_i: G \times G \rightarrow G$ ($i = 1, 2$) denote the multiplication and projection maps. In this paper, we investigate the subgroup
\[
\Ext^1(G, \Gm) := \{ \mathscr{L} \in \Pic(G) \mid m^*\mathscr{L} \simeq \pi_1^*\mathscr{L} \otimes \pi_2^*\mathscr{L}\} \subset \Pic(G)
\]
consisting of line bundles that are (universally) translation-invariant ``modulo the base.'' That is, $\mathscr{L} \in \Ext^1(G, \Gm)$ if and only if for every $k$-scheme $S$ and every $g \in G(S)$, if we let $t_g: G_S \rightarrow G_S$ denote translation by $g$, then $t_g^*\mathscr{L} = \mathscr{L}$ as elements of $\Pic(G_S)/\Pic(S)$. The group $\Ext^1(G, \Gm)$ is in some sense more natural than $\Pic(G)$ because it remembers the group structure of $G$.

Let us explain the reason for the notation $\Ext^1(G, \Gm)$. Consider the collection of extensions of algebraic groups
\[
1 \longrightarrow \Gm \longrightarrow E \longrightarrow G \longrightarrow 1.
\]
Any such extension is necessarily central, because conjugation induces a map $E \rightarrow {\rm{Aut}}_{\Gm/k}$ which must be constant since $E$ is connected and the latter scheme is \'etale. We may therefore make the collection of such extensions modulo the usual notion of equivalence (i.e., there exists a commutative diagram of extensions) into an abelian group via Baer sum, and we temporarily denote this group by $\Ext^1_{\rm{Yon}}(G, \Gm)$. Now any extension $E$ as above is in particular a $\Gm$-torsor over $G$, hence we obtain a map (which one can show to be a homomorphism)
\[
\Ext^1_{\rm{Yon}}(G, \Gm) \rightarrow {\rm{H}}^1(G, \Gm) = \Pic(G).
\]
This map induces an isomorphism $\Ext^1_{\rm{Yon}}(G, \Gm) \xrightarrow{\sim} \Ext^1(G, \Gm)$. (This is essentially \cite[Th.\,4.12]{colliot-thelene}.)
We should also remark that when $G$ is commutative, the above Yoneda Ext group agrees with the derived functor Ext group (Proposition \ref{derivedextagrees}).

The main purpose of the present paper is to prove the following result.

\begin{theorem}
\label{extisfinite}
If $G$ is a connected linear algebraic group over a global function field $k$, then $\Ext^1(G, \Gm)$ is finite.
\end{theorem}

Theorem \ref{extisfinite} is false over every imperfect separably closed field as well as every local function field (Proposition \ref{infiniteextseparablyclosed}), so this result is truly arithmetic in nature. As a corollary of Theorem \ref{extisfinite}, we also obtain the following result.

\begin{theorem}
\label{extisfg}
Let $k$ be a global field, $G$ a smooth connected $k$-group scheme. Then Ext$^1(G, \mathbf{G}_m)$ is a finitely-generated abelian group.
\end{theorem}

One crucial new tool used in the proof of Theorem \ref{extisfinite} is the recent generalization by the author of classical Tate duality to positive-dimensional groups obtained in \cite{rostateduality}. Another is an unpublished result of Ofer Gabber for linear algebraic groups over arbitrary fields which he communicated to the author and whose proof (due to Gabber) we will give here since we shall require it:

\begin{theorem}(O.\,Gabber)
\label{unirationalpic}
Let $k$ be a field, and let $G$ be a smooth connected $k$-group scheme such that $G_{k_s}$ is unirational. Then $\Pic(G) = \Ext^1(G, \Gm)$, and this common group is finite.
\end{theorem}

\begin{remark}
The hypotheses of Theorem \ref{unirationalpic} hold for any connected reductive group $G$ \cite[Chap.\,V, Thm.\,18.2(ii)]{borel}. (In fact, $G_{k_s}$ is even rational, as follows from the open cell decomposition associated to a split maximal torus.) It follows that if $G$ is connected reductive (over any field), then $\Pic(G) = \Ext^1(G, \Gm)$.
\end{remark}

Theorem \ref{unirationalpic} is a slight generalization of an unpublished theorem of Gabber which he communicated to the author (in the original, $G$ was assumed to be unirational over $k$). We give Gabber's elegant proof in \S \ref{sectionunirational}. We remark that the unirationality assumption in Theorem \ref{unirationalpic} implies that $G$ is a connected linear algebraic group, by passing to $\overline{k}$ and using Chevalley's Theorem and the fact that abelian varieties do not admit nonconstant maps from open subsets of affine space.

If $G$ is a connected linear algebraic group over a perfect field $k$, then $\Pic(G) = \Ext^1(G, \Gm)$. This may be deduced from the fact that (thanks to \cite[Lem.\,6.6(i)]{sansuc} and the rationality of linear algebraic groups over algebraically closed fields) for a pair of such groups $G$ and $H$, the map $\pi_G^* + \pi_H^*: \Pic(G) \oplus \Pic(H) \rightarrow \Pic(G \times H)$ is an isomorphism, where $\pi_G: G \times H \rightarrow G$ is the projection and $\pi_H$ is defined similarly. We write that $\Pic(G \times H) = \Pic(G) \oplus \Pic(H)$. Over imperfect fields, this equality fails in general (Proposition \ref{pic=/=ext}). Our next result, however, shows that under certain hypotheses we do recover this nice behavior even over imperfect fields.

\begin{theorem}
\label{picismultiplicativeforgroupswithfinitepic}
Let $G$ be a connected linear algebraic group over a field $k$, and assume that $\Pic(G)$ is finite.
\begin{itemize}
\item[(i)] If $G(k)$ is Zariski dense in $G$, then $\Pic(G) = \Ext^1(G, \Gm)$.
\item[(ii)] If $k$ is separably closed, then $\Pic(G) = \Ext^1(G, \Gm)$.
\end{itemize}
\end{theorem}

In fact, we prove a somewhat more general result; see Theorem \ref{Pic(GXH)}. The hypothesis that $G(k)$ is Zariski dense in $G$ is necessary; see Example \ref{zariskidenseneeded}.

Finally, we present applications of our study of Picard groups to the cohomology of algebraic groups over local and global fields, and especially to constructing counterexamples of various types. In particular, we show that, in contrast to the number field setting, for a connected linear algebraic group $G$ over a global function field $k$, and a finite nonempty set $S$ of places of $k$, the $S$-Tate-Shafarevich set
\[
\Sha_S(G) := \ker\left({\rm{H}}^1(k, G) \longrightarrow \prod_{v \notin S} {\rm{H}}^1(k_v, G)\right)
\]
can be infinite (Corollary \ref{formGainfinitesha}). This is impossible when $S = \emptyset$ by \cite[Thm. 1.3.3(i)]{conrad}. (The original version of that result incorrectly stated that the same holds for any finite $S$; see Remark \ref{conraderrata}.)
\newline

\noindent {\bf Acknowledgments.} I would like to thank Brian Conrad for helpful comments and suggestions, and Ofer Gabber for telling the author about Theorem \ref{unirationalpic} and its proof, which both simplified and generalized the original version of this paper.

\noindent {\bf Notation and Conventions.} Throughout this paper, $k$ denotes a field, and $k_s$, $k_{{\rm{per}}}$, and $\overline{k}$ denote separable, perfect, and algebraic closures of $k$, respectively. When it appears, $p$ denotes the characteristic of $k$. A {\em linear algebraic group} over $k$ (or linear algebraic $k$-group) is a smooth affine $k$-group scheme. If we have two $k$-schemes $X$ and $Y$, then we write $\Pic(X \times Y) = \Pic(X) \oplus \Pic(Y)$ if the map $\pi_X^* + \pi_Y^*: \Pic(X) \oplus \Pic(Y) \rightarrow \Pic(X \times Y)$ is an isomorphism, where $\pi_X: X \times Y \rightarrow X$ denotes the projection map, and similarly for $\pi_Y$. (If $X(k)$ and $Y(k)$ are nonempty -- e.g., if $X$ and $Y$ are $k$-group schemes -- then this map is injective, so the real issue is the surjectivity.)

\section{Groups with finite Pic}
\label{sectionfinitepic}

If $G$ and $H$ are connected linear algebraic groups over a field $k$ such that $G_{k_s}$ is rational, then Pic$(G \times H) = {\rm{Pic}}(G) \oplus {\rm{Pic}}(H)$ \cite[Lem.\,6.6(i)]{sansuc}. In particular, this holds over perfect fields (since any connected linear algebraic group over an algebraically closed field is rational) and applied with $G = H$ yields the equality $\Pic(G) = \Ext^1(G, \Gm)$. The goal of this section is to prove Theorem \ref{picismultiplicativeforgroupswithfinitepic}, which roughly says that in this respect groups with finite Pic behave as if they were rational (or as if the underlying field were perfect). In fact, by analogy with the situation with perfect fields, we prove the more general equality $\Pic(G \times H) = \Pic(G) \oplus \Pic(H)$ under suitable hypotheses on $G$ and $H$ (Theorem \ref{Pic(GXH)} below). 

For an example showing the necessity of the hypothesis that $G(k)$ is Zariski dense in $G$ in Theorem \ref{picismultiplicativeforgroupswithfinitepic}(i) (and therefore the necessity of the hypothesis that $H(k)$ is Zariski dense in $H$ in Theorem \ref{Pic(GXH)}(i)), see Example \ref{zariskidenseneeded}. We begin with the following simple lemma.

\begin{lemma}
\label{effective}
Let $X$ be a Noetherian normal (i.e., all local rings of $X$ are normal) affine scheme, and let $Z \subset X$ be a nonempty closed subscheme. Every $\mathscr{L} \in {\rm{Pic}}(X)$ is of the form $\mathscr{L} = \mathcal{O}(D)$ for some effective Weil divisor $D$ on $X$ not containing $Z$.
\end{lemma}

\begin{proof}
This argument is a slightly modified version of an argument suggested by Ofer Gabber. We may assume (by replacing $Z$ with one of its irreducible components and then replacing it with its reduced subscheme) that $Z$ is integral. Let $\eta$ be the generic point of $Z$. Let $U \subset X$ be a dense open subset containing $\eta$ such that $\mathscr{L}|_U$ is trivial. We can find such $U$ by choosing it to be a union of disjoint open neighborhoods of $\eta$ and the generic points of the irreducible components of $X$ not containing $Z$ such that $\mathscr{L}$ trivializes on each such neighborhood. Then if we choose a generating section of $\mathscr{L}|_U$, its divisor $D$ does not contain $\eta$ (hence does not contain $Z$), and $\mathscr{L} \simeq \calO(D)$. The only problem is that $D$ may not be effective.

To remedy this, let $\mathfrak{p}_1,..., \mathfrak{p}_m \subset A := \Gamma(X, \mathcal{O}_X)$ be the codimension-one primes appearing with negative multiplicity in $D$, and let $\mathfrak{q}_Z$ be the prime corresponding to $Z$. Since $Z$ is not contained in supp$(D)$, we have $\mathfrak{p}_i \not \subset \mathfrak{q}_Z$ for each $i$. We may therefore choose $a \in (\mathfrak{p}_1 \cap ... \cap \mathfrak{p}_m) - \mathfrak{q}_Z$. (Indeed, for each $i = 1, \dots, m$, choose $a_i \in \mathfrak{p}_i - \mathfrak{q}_Z$. Then $a := a_1 \dots a_m \in (\mathfrak{p}_1 \cap ... \cap \mathfrak{p}_m) - \mathfrak{q}_Z$ because $\mathfrak{q}_Z$ is prime.) Then replacing $D$ with $D + m\cdot \mbox{div}(a)$ for sufficiently large $m$ yields an effective linearly equivalent divisor still not containing $Z$. This proves the lemma.
\end{proof}

The following lemma is the key to proving Theorem \ref{picismultiplicativeforgroupswithfinitepic}.

\begin{lemma}
\label{pictrivialfibers}
Let $G$ and $H$ be connected linear algebraic groups over a field $k$, and let $\mathscr{L} \in {\rm{Pic}}(G \times H)$. Suppose that there is $g \in G(k)$ such that $\mathscr{L}|_{\{g\} \times H} = 0$, and there is a Zariski dense set of rational points $S_H \subset H(k)$ such that $\mathscr{L}|_{G \times \{h\}} = 0$ for every $h \in S_H$. Then $\mathscr{L}$ is trivial.
\end{lemma}

\begin{proof}
By Lemma \ref{effective}, there is an effective divisor $D$ on $G \times H$ such that $\{g\} \times H \not \subset \mbox{supp}(D)$ and such that $\mathscr{L} \simeq \mathcal{O}(D)$. We need to show that $D = \mbox{div}(F)$ for some $F \in k[G \times H]$. We will abuse notation and also use the notation $D$ for the divisor $D_{k_{\rm{per}}}$ on $G_{k_{\rm{per}}} \times_{k_{\rm{per}}} H_{k_{\rm{per}}}$.

We have $\Pic(G_{k_{\rm{per}}} \times H_{k_{\rm{per}}}) = \Pic(G_{k_{\rm{per}}}) \oplus \Pic(H_{k_{\rm{per}}})$. Our assumptions therefore imply that there exists $f \in k_{\rm{per}}[G \times H]$ such that $\mbox{div}(f) = D$. Our assumptions also imply that $f|_{\{g\} \times H} \in k_{\rm{per}}[H]^{\times}\cdot k[H]$ and $f|_{G \times \{h\}} \in k_{\rm{per}}[G]^{\times} \cdot k[G]$ for all $h \in S_H$.

We have $k_{\rm{per}}[G]^{\times} = k_{\rm{per}}^{\times} \cdot \widehat{G}(k_{\rm{per}})$ by Rosenlicht's Unit Theorem. Since $\widehat{G}(k_{\rm{per}})/\widehat{G}(k)$ is finite (Lemma \ref{finitequotientchars}), there is some $\chi \in \widehat{G}(k_{\rm{per}})$ and some Zariski dense subset $S_H' \subset S_H$ such that $f|_{G \times \{h\}} \in \chi \cdot k_{\rm{per}}^{\times} \cdot k[G]$ for every $h \in S_H'$. (Here we are using the connectedness of $H$.) Thus, dividing $f$ by $\chi \in k[G]^{\times} \subset k[G \times H]^{\times}$, which doesn't affect div$(f)$, and renaming $S_H'$ as $S_H$, we may assume that $f|_{G \times \{h\}} \in k_{\rm{per}}^{\times} \cdot k[G]$ for every $h \in S_H$. Similarly (actually, even more simply), by dividing $f$ by some element of $\widehat{H}(k_{\rm{per}})$, we may also assume that $f|_{\{g\} \times H} \in k_{\rm{per}}^{\times} \cdot k[H]$. Under these assumptions, we claim that $f \in k_{\rm{per}}^{\times} \cdot k[G \times H]$. This will complete the proof because we then have $f = \beta f'$ for some $\beta \in k_{{\rm{per}}}^{\times}$ and $f' \in k[G \times H]$, and we then have ${\rm{div}}(f) = {\rm{div}}(f')$, so that $\mathscr{L} \simeq \mathcal{O}({\rm{div}}(f)) \simeq 0 \in \Pic(G \times H)$.

Since $\{g\} \times H \not \subset {\rm{supp}}(D)$, we have $f|_{\{g\} \times H} \neq 0$.
Then shrinking $S_H$ if necessary, we have $f(g, h) \neq 0$ for each $h \in S_H$. By hypothesis, there is $\lambda \in k_{\rm{per}}^{\times}$ such that $\lambda^{-1} f|_{\{g\} \times H} \in k[H]$. Replacing $f$ with $f/\lambda$, we then have $f(g, h) \in k^{\times}$ for each $h \in S_H$. Then for each $h \in S_H$, we have $f|_{G \times \{h\}} \in k_{\rm{per}}^{\times} \cdot k[G]$, but since $f(g, h) \in k^{\times}$, we see that in fact $f|_{G \times \{h\}} \in k[G]$.

We will now show that $f \in k[G \times H]$. To prove this, let $V := k[G \times H]$. We have $f \in k_{\rm{per}} \otimes_k V = k_{\rm{per}}[G \times H]$, and we want to show that $f$ actually comes from an element of $V$. By faithfully flat descent, this is equivalent to showing that the images of $f$ under the two projections $V_{k_{\rm{per}}} \rightrightarrows V_{(k_{\rm{per}} \otimes_k k_{\rm{per}})}$ are equal. Consider the difference between these two images; call it $\alpha$. We want to show that $\alpha = 0$. Since $f|_{G \times \{h\}} \in k[G]$ for every $h \in S_H$, we know that $\alpha|_{G \times \{h\}} \in V_{(k_{\rm{per}} \otimes_k k_{\rm{per}})}|_{G \times \{h\}} = (k_{\rm{per}} \otimes_k k_{\rm{per}})[G]$ is zero. What we must show, therefore, is that any such $\alpha \in V_{(k_{\rm{per}}\otimes_k k_{\rm{per}})}$ is $0$. Choose a $k$-basis $\{e_i\}$ of $k_{\rm{per}} \otimes_k k_{\rm{per}}$. Then we may uniquely write $\alpha = \sum e_i \alpha_i$ for some $\alpha_i \in V$, and we have that $\alpha_i|_{G \times \{h\}} = 0$ for every $i$ and every $h \in S_H$. What we must show, therefore, is that if $A \in V$ is such that each $A|_{G \times \{h\}}$ (with $h \in S_H$) is $0$, then $A = 0$.

For this, choose a $k$-basis $\{w_i\}$ of $k[G]$. We may write $A = \sum_i z_iw_i$ for some $z_i \in k[H]$. We know that $z_i(h) = 0$ for every $h \in S_H$. Since $S_H$ is Zariski dense in $H$, this implies that $z_i = 0$. Since this holds for each $w_i$, we see that $A = 0$. This completes the proof of the lemma.
\end{proof}

Taking $G = H$ in the following theorem yields Theorem \ref{picismultiplicativeforgroupswithfinitepic}.

\begin{theorem}
\label{Pic(GXH)}
Let $G$ and $H$ be connected linear algebraic groups over a field $k$, and suppose that $\Pic(G)$ is finite.
\begin{itemize}
\item[(i)] If $H(k)$ is Zariski dense in $H$, then $\Pic(G \times H) = \Pic(G) \oplus \Pic(H)$.
\item[(ii)] If $k$ is separably closed, then $\Pic(G \times H) = \Pic(G) \oplus \Pic(H)$.
\end{itemize}
\end{theorem}

\begin{proof}
Part (ii) follows from (i) and the fact that $H(k)$ is Zariski dense in the smooth group $H$ if $k$ is separably closed, so we concentrate on (i). Let $\mathscr{L} \in {\rm{Pic}}(G \times H)$. Since ${\rm{Pic}}(G)$ is finite and $H(k)$ is Zariski dense in $H$, there is a line bundle $\mathscr{L}_1 \in {\rm{Pic}}(G)$ and a Zariski dense set $S_H \subset H(k)$ such that $\mathscr{L}|_{G \times \{h\}} \simeq \mathscr{L}_1$ for all $h \in S_H$. (Here we have made use of the connectedness of $H$ to conclude that if a finite union of sets of rational points is Zariski dense, then one of the sets is Zariski dense.) Let $\mathscr{L}_2 := \mathscr{L}|_{\{1\} \times H} \in {\rm{Pic}}(H)$. By Lemma \ref{pictrivialfibers} (with $g = 1$), $\mathscr{L} \otimes \pi_G^*\mathscr{L}_1^{-1} \otimes \pi_H^*\mathscr{L}_2^{-1}$ is trivial. Since $\mathscr{L} \in {\rm{Pic}}(G \times H)$ was arbitrary, this proves that the map $\Pic(G) \oplus \Pic(H) \rightarrow \Pic(G \times H)$, which is always injective (as is seen by restricting to $1 \times H$ and $G \times 1$), is surjective as well under the hypotheses of the Theorem.
\end{proof}

\section{Unirational groups}
\label{sectionunirational}

In this section we will give Gabber's proof of Theorem \ref{unirationalpic}. (All results and proofs in this section are due to Gabber.) The first step is to show that $\Pic(G) = \Ext^1(G, \Gm)$:

\begin{proposition}
\label{pic=extunirational}
If $G$ is a smooth group scheme over a field $k$ and $G_{k_s}$ is unirational, then $\Pic(G) = \Ext^1(G, \Gm)$.
\end{proposition}

\begin{proof}
We first construct an open immersion of $k$-schemes $G \hookrightarrow G^c$ with $G^c$ projective and normal such that the multiplication map $G \times G \rightarrow G$ extends to a map $G \times G^c \rightarrow G^c$. We accomplish this as follows. If $G = {\rm{GL}}_n$, then we take the inclusion $G \hookrightarrow \A^{n^2} \hookrightarrow \mathbf{P}^{n^2}$, and one easily checks that the multiplication map extends in this case. In general, we embed $G$ into some ${\rm{GL}}_n$ and take its Zariski closure $G'$ in $\mathbf{P}^{n^2}$. The ``multiplication'' map ${\rm{GL}}_n \times \mathbf{P}^{n^2} \rightarrow \mathbf{P}^{n^2}$ then restricts to a map $G \times G' \rightarrow G'$ which extends the multiplication map on $G$. Now we take the normalization $G^c$ of $G'$, and the above map then lifts to a map $G \times G^c \rightarrow G^c$ by the universal property of normalization.

Now let $\mathscr{L} \in \Pic(G)$. Extend $\mathscr{L}$ to a rank-one reflexive sheaf $\mathscr{L}^c$ on the normal scheme $G^c$. (This is possible because of the isomorphism between rank-one reflexive sheaves and Weil divisor class groups on normal Noetherian schemes; see \cite[\S 3]{schwede}.) Define the group-valued fppf sheaf $E$ on smooth $k$-schemes by the formula
\[
E(S) := \{ g \in G(S), \phi: t_g^*\mathscr{L}^c_S \xrightarrow{\sim} \mathscr{L}^c_S\},
\]
where $t_g: G^c_S \rightarrow G^c_S$ is left-translation by $g$ and the group law on $E$ is the evident one: given two pairs $(g, \phi), (g', \phi')$ as above, their product is the pair $(gg', \psi)$, where $\psi$ is the composition
\[
t_{gg'}^*\mathscr{L}^c_S = t_{g'}^*t_g^*\mathscr{L}^c_S \xrightarrow[\sim]{t_{g'}^*\phi} t_{g'}^*\mathscr{L}^c_S \xrightarrow[\sim]{\phi'} \mathscr{L}^c_S.
\]

Let $\pi: E \rightarrow G$ denote the obvious map. We claim that $\ker(\pi) = \Gm$. Indeed, $\ker(\pi) = {\rm{Isom}}(\mathscr{L}^c, \mathscr{L}^c)$. But we have $\calHom(\mathscr{L}^c, \mathscr{L}^c) = \calO_{G^c}$ as coherent sheaves on $G^c$, as holds for any rank-one reflexive sheaf on a normal scheme by using the fact that any such sheaf is a subsheaf of the sheaf of rational functions. Then we deduce that the same equality holds after base change since Hom sheaves commute with flat base change. Thus, in order to prove that $\ker(\pi) = \Gm$, we only need to show that $\Gamma(G^c_S, \Gm) = \Gamma(S, \Gm)$ for any $k$-scheme $S$, for which it suffices to show that $\Gamma(G^c_S, \calO_{G^c_S}) = \Gamma(S, \calO_S)$. Since coherent cohomology commutes with flat base change, we only need to check this for $S = {\rm{Spec}}(k)$. In that case, $\Gamma(G^c, \calO_{G^c})$ is a finite $k$-algebra because $G^c$ is proper, and because $G^c$ is integral, it must therefore be a finite field extension of $k$. But given any such global section, its restriction to $G$ must be an element of $k$ since $k$ is algebraically closed in the function field of $G$. Since $G^c$ is integral, therefore, any such global section must in fact lie in $k$. This completes the proof that $\ker(\pi) = \Gm$.

Now we claim that the map $\pi: E \rightarrow G$ of fppf sheaves is surjective. This is the key point in the proof where we use the fact that $G_{k_s}$ is unirational. Indeed, by assumption there is a dominant morphism of $k_s$-schemes $g: U \rightarrow G_{k_s}$ such that $U$ is an open subscheme of some affine space. By generic flatness, we may assume after shrinking $U$ that $g$ is flat. A finite list of translates of $g$ then cover $G$, so that we have a flat surjective map $g: U^n \rightarrow G$ for some $n > 0$. Therefore, renaming $U^n$ as $U$, we obtain that $g$ is surjective, so $g$ is an fppf cover (i.e., a surjective fppf map). Standard direct limit arguments imply that $U$ descends to a $k'$-scheme for some finite separable extension $k'/k$, and that $g$ also so descends. In order to check the surjectivity of $\pi$, it suffices to check it after extending scalars to $k'$. By translating, we may also assume that $u_0 \mapsto 1 \in G(k')$ for some $u_0 \in U(k')$.

Note that $G^c_{k'}$ is a normal scheme, since $k'/k$ is \'etale. Using the isomorphism between the group of rank-one reflexive sheaves and the Weil divisor class group on a Noetherian normal scheme, together with the equality ${\rm{Cl}}(U \times G^c_{k'}) = {\rm{Cl}}(U) \oplus {\rm{Cl}}(G^c_{k'})$ which uses the normality of $G^c_{k'}$ \cite[Cor. 21.4.11]{ega}, we deduce from the vanishing of ${\rm{Cl}}(U)$ that $t_g^*\mathscr{L}^c \simeq \mathscr{M}$ for some rank-one reflexive sheaf $\mathscr{M}$ on $G^c$. Restricting to $u_0$ shows that $\mathscr{M} \simeq \mathscr{L}^c$. We therefore deduce (from the definition of $E$) that $g$ factors through a map $U \rightarrow E$. Thus, upon pullback to the fppf cover $U \rightarrow G$, $\pi$ obtains a section, hence it is fppf surjective as claimed. 

Since $E \rightarrow G$ is surjective with kernel $\Gm$, the effectivity of fppf descent for relatively affine schemes implies that $E$ is representable by a finite type $k$-group scheme. Fix an isomorphism $e^*\mathscr{L} \xrightarrow{\sim} k$, where $e \in G(k)$ is the identity. Any $S$-valued point of $E$ yields an isomorphism $t_g^*\mathscr{L}^c_S \xrightarrow{\sim} \mathscr{L}^c_S$. Restricting to the identity of $G$ then yields an isomorphism $g^*\mathscr{L} \xrightarrow{\sim} \calO_S$. We therefore have a map $E \rightarrow {\rm{Isom}}(\mathscr{L}, \calO)$, where the latter sheaf ``is'' the line bundle $\mathscr{L}$ (i.e., it is the corresponding $\Gm$-torsor; the Isom here is in the category of fppf sheaves of $\mathcal{O}_G$-modules on $G$). Since this is a map of $\Gm$-torsors over $G$, it is an isomorphism. We deduce that $\mathscr{L}$ arises (as a $\Gm$-torsor) from an extension of $G$ by $\Gm$, which proves the proposition.
\end{proof}

Before turning to the proof of Theorem \ref{unirationalpic}, we turn to the following special case which is worthy of note in its own right.

\begin{proposition}
\label{perfectgpspic}
Let $G$ be a connected linear algebraic group over a field $k$, and suppose that $G = \mathscr{D}G$. Then $\Pic(G) = \Ext^1(G, \Gm)$, and this group is finite.
\end{proposition}

\begin{proof}
The key point is that any perfect (i.e., equal to its own derived group) connected linear algebraic group $G$ over a field $k$ is unirational over $k_s$ (in fact, even over $k$ but we do not need this). Indeed, by \cite[Prop.\,A.2.11]{cgp}, $G$ is generated by its $k$-tori, hence there is a surjection of $k$-schemes $\prod_i T_i \twoheadrightarrow G$ from a finite product of tori to $G$. The unirationality claim then follows from the fact that tori over $k_s$ are unirational. (In fact, tori are unirational over any field but we do not require this.)

By Proposition \ref{pic=extunirational}, therefore, $\Pic(G) = \Ext^1(G, \Gm)$. To show that the latter group is finite, we claim that the map $\Ext^1(G, \Gm) \rightarrow \Ext^1_{\overline{k}}(G, \Gm)$ is injective. Assuming this, the finiteness of the latter group (since $\overline{k}$ is perfect) implies that of the former. To see the desired injectivity, we note that the kernel of this map consists of the $k$-forms of the trivial extension. Since the automorphism functor of the trivial extension of $G$ by $\Gm$ (as an extension) is $\calHom(G, \Gm) = 0$ (since $G = \mathscr{D}G$), it follows that this kernel is trivial (since it is classified by ${\rm{H}}^1(k, {\rm{Aut}}_{\rm{triv}})$, where ${\rm{Aut}}_{\rm{triv}}$ denotes the automorphism functor of the trivial extension).
\end{proof}

\begin{lemma}
\label{extexactsequence}
Given a short exact sequence
\begin{equation}
\label{extension1}
1 \longrightarrow G' \longrightarrow G \longrightarrow G'' \longrightarrow 1
\end{equation}
of finite type $k$-group schemes such that $G$ $($hence also $G''$$)$ is smooth and connected, the resulting sequence
\[
0 \longrightarrow \widehat{G''}(k) \longrightarrow \widehat{G}(k) \longrightarrow \widehat{G'}(k) \longrightarrow \Ext^1(G'', \Gm) \longrightarrow \Ext^1(G, \Gm)
\]
is exact, where all of the maps are given by pullback, except for $\widehat{G'}(k) \rightarrow \Ext^1(G'', \Gm)$, which sends a character $\chi$ to the pushout of the extension $($$\ref{extension1}$$)$ along $\chi$. If $G'$ is also smooth and connected, then this exact sequence may be extended to include the map $\Ext^1(G, \Gm)  \rightarrow \Ext^1(G', \Gm)$.
\end{lemma}

\begin{proof}
Exactness is clear at $\widehat{G''}(k)$ and $\widehat{G}(k)$. To check exactness at $\widehat{G'}(k)$, suppose given $\chi \in \widehat{G'}(k)$. Then to say that the pushout of (\ref{extension1}) along $\chi$ has a section to $\Gm$ is the same, by the universal property of pushout, as saying that $\chi$ extends to a character of $G$.

For exactness at $\Ext^1(G'', \Gm)$, given an extension $E''$ of $G''$ by $\Gm$, let $E$ denote its pullback to $G$ and suppose that $E$ admits a section $s: G \rightarrow E$. Then the composition $\phi: G \xrightarrow{s} E \rightarrow E''$ is a map such that its composition with $f: E'' \rightarrow G''$ is the map $G \rightarrow G''$. In particular, $f \circ \phi$ kills $G'$, hence $\phi$ restricts to character $\chi$ of $G'$, and one may check that $E''$ is the pushout of (\ref{extension1}) along $\chi$.

Finally, let us check exactness at $\Ext^1(G, \Gm)$. Let $E$ be an extension of $G$ by $\Gm$ such that the preimage of $G'$ admits a section $s: G' \rightarrow E$. We claim that $s(G') \subset E$ is normal. Let $f$ denote the map $E \rightarrow G$. The conjugation action of $E$ on $s(G')$ yields a map $\phi: E \rightarrow \calHom(G', \Gm)$ defined by the formula
\[
es(g')e^{-1} = s(f(e)g'f(e)^{-1})\phi_e(g')
\]
for $e \in E$, $g' \in G'$, where $\phi_e$ denotes $\phi(e) \in \calHom(G', \Gm)$. We need to show that $\phi$ is trivial. It suffices to check this over $\overline{k}$, so we may assume that $k$ is algebraically closed for the purposes of checking this claim. The map $\phi$ descends to a map $E \rightarrow \calHom(M', \Gm)$, where $M'$ is the maximal multiplicative quotient of $G'$. Indeed, for any $e \in E(k)$ the map $\phi_e$ so descends (since over a field any map $G' \rightarrow \Gm$ descends to the maximal multiplicative quotient), hence so does $\phi$ since $E(k)$ is Zariski dense in $E$ (as $E$ is reduced). Finally, $\calHom(M', \Gm)$ is \'etale, so the map from the connected $E$ to it must be trivial.

We therefore see that $s(G') \subset E$ is normal, so we may form the quotient $E/s(G')$, which is an extension of $G''$ by $\Gm$ (via the composition $E \rightarrow G \rightarrow G''$) that pulls back to $E$.
\end{proof}

\begin{proof}[Proof of Theorem $\ref{unirationalpic}$]
Thanks to Proposition \ref{pic=extunirational}, it only remains to prove the finiteness assertion. If $G$ is unirational over $k_s$, then so are $\mathscr{D}G$ and $G^{\rm{ab}} := G/\mathscr{D}G$. Indeed, for $G^{\rm{ab}}$ this is clear, while for $\mathscr{D}G$ it follows from the fact that this group is the image of a map $G^n \rightarrow G$ for some $n$. (This map consists of a product of commutator maps \cite[Chap.\,I, Prop.\,2.2]{borel}.) By Lemma \ref{extexactsequence}, we may therefore repeatedly filter $G$ by $\mathscr{D}G$ and $G^{\rm{ab}}$ and thereby assume that $G$ is either commutative or else $G = \mathscr{D}G$. The latter case is handled by Proposition \ref{perfectgpspic}.

To handle the case when $G$ is commutative, we note that unirationality of $G$ implies that $G$ is generated by a finite collection of maps $U \rightarrow G$ from open subsets of the affine line. For any effective divisor $D$ on $\mathbf{P}^1$, define the generalized Jacobian $J_D$ functorially by the formula $J_D(S) := \{ \mathscr{L} \in \Pic^0(\mathbf{P}^1_S), \phi: \mathscr{L}_D \xrightarrow{\sim} \calO_D\}$. By \cite[Ch.\,I. \S 1, Thm.\,1]{serreclassfields}, for any map $f: U \rightarrow G$, there is an effective divisor $D$ on $\mathbf{P}^1$ disjoint from $U$ such that $f$ factors as a composition $U \rightarrow J_D \rightarrow G$, where the first map is the one coming from $u \mapsto [u] - [u_0]$ for some $u_0 \in U(k)$, and the second is a $k$-group homomorphism. (The proof in \cite{serreclassfields} does not assume perfection of $k$ in the case when $G$ is affine. It is used, however, in the case when $G$ is neither affine nor proper in order to use Chevalley's Theorem on the structure of algebraic groups (though this assumption could probably be eliminated, either by using the fact that Chevalley applies after a finite purely inseparable extension or by applying the structure theorem \cite[Thm.\,A.3.9]{cgp} when ${\rm{char}}(k) > 0$).) 

Since $\Jac(\mathbf{P}^1) = 0$, we have $J_D = {\rm{Aut}}(\mathcal{O}_D)$, the automorphism functor of $\mathcal{O}_D$. Therefore,  $J_D(S) = \Gamma(D_S, \mathcal{O}_{D_S}^{\times})$ for $k$-schemes $S$. That is, $J_D \simeq \R_{D/k}(\Gm)$, where $\R_{D/k}$ denotes Weil restriction of scalars. Since $G$ is generated (as a group) by the images of a finite collection of maps $U \rightarrow G$ from open subsets of the affine line, it follows that $G$ fits into an exact sequence
\[
0 \longrightarrow H \longrightarrow \R_{A/k}(\Gm) \longrightarrow G \longrightarrow 1
\]
for some $k$-group scheme $H$ and some finite $k$-algebra $A$. Since $\R_{A/k}(\Gm)$ is an open subscheme of some affine space, it has trivial Picard group and therefore trivial $\Ext^1(\cdot, \Gm)$. Lemmas \ref{extexactsequence} and \ref{finitecokernelchars} therefore complete the proof.
\end{proof}

\section{Finiteness of $\Ext^1(G, \Gm)$ over global function fields}
\label{completingtheproof}

In this section we will prove Theorems \ref{extisfinite} and \ref{extisfg}. Before proving Theorem \ref{extisfinite}, however, let us prove a related finiteness result over general fields which says that, in general, the potential infinitude of $\Ext^1(G, \Gm)$ comes from the case when $G$ is a form of $\Ga$.

\begin{proposition}
\label{extfiniteformsofGa}
Let $k$ be a field such that $\Ext^1(U, \Gm)$ is finite for every $k$-form $U$ of $\Ga$. Then $\Ext^1(G, \Gm)$ is finite for every connected linear algebraic $k$-group $G$.
\end{proposition}

\begin{proof}
By Lemma \ref{extexactsequence}, if the desired finiteness holds for the two end terms in a short exact sequence, then it also holds for the middle term. We will repeatedly use this fact below. By filtering $G$ between its derived group and its abelianization, we may assume that $G$ is either perfect or commutative. The perfect case is handled by Proposition \ref{perfectgpspic}, so we may assume that $G$ is commutative.

So let $G$ be a connected commutative linear algebraic $k$-group, and let $T \subset G$ be its maximal torus (which is unique because $G$ is commutative). Then $U := G/T$ is unipotent, so we may assume that $G$ is either a torus or commutative unipotent. In the former case, $G = T$ is unirational over $k_s$ (even over $k$), so we are done by Theorem \ref{unirationalpic}, which was proved in \S \ref{sectionunirational} (though for tori this finiteness is much easier to prove).

To handle the unipotent case, we claim that any smooth connected unipotent $k$-group $U$ admits a filtration by forms of $\Ga$. This will prove the proposition. In order to prove this claim, we may assume that $U$ is commutative (as we may for our purposes anyhow), and since $U$ is killed by some power of $p := {\rm{char}}(k)$, we may even assume that $U$ is killed by $p$, by making use of dimension induction and the exact sequence
\[
0 \longrightarrow [p]U \longrightarrow U \longrightarrow U/[p]U \longrightarrow 0,
\]
where $[p]: U \rightarrow U$ is the multiplication by $p$ map. By \cite[Lem.\,B.1.10]{cgp}, there is a finite \'etale $k$-subgroup $E \subset U$ such that $\overline{U} := U/E$ is split unipotent. Let $\pi: U \rightarrow \overline{U}$ denote the quotient map. Since $\overline{U}$ is split, there is a filtration $0 = \overline{U}_0 \subset \overline{U}_1 \subset \dots \overline{U}_n = \overline{U}$ with $\overline{U}_{i+1}/\overline{U}_i \simeq \Ga$. Let $U_i := \pi^{-1}(\overline{U}_i)^0$. Then $U_i$ is smooth, connected, and unipotent of dimension $i$, so $U_{i+1}/U_i$ is a $k$-form of $\Ga$. This proves our claim and the proposition.
\end{proof}

Although we could proceed more directly towards the proof of Theorem \ref{extisfinite}, due to lack of an alternative reference, we next prove that for a smooth connected commutative group scheme $G$ over a field $k$, the groups $\Ext^1(G, \Gm)$ and the usual derived functor Ext group, which we temporarily denote by $\Ext^1_{\rm{func}}(G, \Gm)$, are canonically isomorphic. Here, the derived functor is taken in the category of fppf abelian sheaves, but we may also simply take the Yoneda Ext group in the category of commutative finite-type $k$-group schemes, since any fppf sheaf extension of $G$ by $\Gm$ is necessarily representable by the effectivity of descent for relatively affine schemes.

\begin{lemma}
\label{commextension}
If $G$ is a smooth connected commutative group scheme (not necessarily affine) over a field $k$, then any extension
\[
1 \longrightarrow T \longrightarrow E \longrightarrow G \longrightarrow 1
\]
with $T$ a torus is commutative.
\end{lemma}

\begin{proof}
We may extend scalars and thereby assume that $k = \overline{k}$. Let $\pi$ denote the map from $E$ to $G$, and let $T' \subset G$ denote the maximal torus. Then $\pi^{-1}(T')$ is a torus. Replacing $T$ with $\pi^{-1}(T')$, therefore, we may assume that $G$ contains no nontrivial torus.

Consider the conjugation map $\phi: E \times E \rightarrow T$ given by $(e_1, e_2) \mapsto e_1e_2e_1^{-1}e_2^{-1}$. This is bi-additive and lands in $T$ because $G$ is commutative. We need to show that $\phi$ is trivial. Since $G$, hence $E$, is connected, the conjugation map $E \rightarrow {\rm{Aut}}_{T/k}$ from a connected $k$-group to an \'etale one is trivial, hence $T \subset E$ is central. (To see that ${\rm{Aut}}_{T/k}$ is \'etale, by Galois descent we may assume that $k = k_s$, and in that case it is just the constant group scheme ${\rm{GL}}_n(\Z)$, where $T_{k_s} \simeq \Gm^n$.) Therefore, $\phi$ descends to a biadditive pairing $G \times G \rightarrow T$ which we still denote by $\phi$.

Since $k = \overline{k}$ and $G$ contains no nontrivial torus, Chevalley's Theorem implies that $G$ is an extension of an abelian variety by a smooth connected unipotent group $U$.
For every $g \in G(k)$, $\phi|_{\{g\} \times U}$ is a $k$-homomorphism $U \rightarrow T$, which must be trivial because there are no nontrivial homomorphisms from a unipotent group scheme to a torus over a field. Since $k = \overline{k}$ and $G$ is smooth, $G(k)$ is Zariski dense in $G$. It therefore follows that $\phi$ is trivial on $G \times U$, and it is similarly trivial on $U \times G$. It therefore descends to a pairing $\phi: A \times A \rightarrow T$, where $A := G/U$ is an abelian variety. But since $A \times A$ is proper, it follows that $\phi$ must be trivial, as desired.
\end{proof}

Lemma \ref{commextension} implies that if $G$ is a smooth connected commutative group over a field $k$, then we obtain a natural homomorphism
\[
\Ext^1(G, \Gm) \longrightarrow \Ext^1_{\rm{func}}(G, \Gm).
\]
Yoneda's Lemma and the effectivity of fppf descent for relatively affine schemes imply that this map is an isomorphism:

\begin{proposition}
\label{derivedextagrees}
If $G$ is a smooth connected commutative group scheme over a field $k$, then the canonical map $\Ext^1(G, \Gm) \rightarrow \Ext^1_{\rm{func}}(G, \Gm)$ defined above is an isomorphism.
\end{proposition}

We now turn to the proof of Theorem \ref{extisfinite}.

\begin{proof}[Proof of Theorem $\ref{extisfinite}$]
By Proposition \ref{extfiniteformsofGa}, we may assume that $G = U$ is a $k$-form of $\Ga$. By Proposition \ref{derivedextagrees}, $\Ext^1(U, \Gm)$ may be interpreted as the usual derived functor Ext. Then \cite[Cor.\,2.3.4]{rostateduality} says that we have a natural isomorphism ${\rm{H}}^1(k, \widehat{U}) \simeq \Ext^1(U, \Gm)$, where $\widehat{U}$ denotes the fppf dual sheaf $\calHom(U, \Gm)$. So we need to show that if $k$ is a global function field, then the group ${\rm{H}}^1(k, \widehat{U})$ is finite for any $k$-form $U$ of $\Ga$. 

For an abelian group $B$, let $B^* := {\rm{Hom}}(B, \Q/\Z)$ denote the $\Q/\Z$ dual group.
Positive-dimensional Tate duality \cite[Thm.\,1.2.8]{rostateduality} furnishes an exact sequence
\[
{\rm{H}}^1(k, U) \longrightarrow {\rm{H}}^1(\A, U) \longrightarrow {\rm{H}}^1(k, \widehat{U})^* \longrightarrow {\rm{H}}^2(k, U),
\]
where $\A$ denotes the ring of adeles of $k$. Since $U$ is unipotent, ${\rm{H}}^2(k, U) = 0$ \cite[Prop.\,2.5.4(i)]{rostateduality}. Also, since $U$ is smooth and connected, the natural map ${\rm{H}}^1(\A, U) \rightarrow \prod_v {\rm{H}}^1(k_v, U)$ identifies ${\rm{H}}^1(\A, U)$ with $\oplus_v {\rm{H}}^1(k_v, U)$ \cite[Prop.\,5.2.2]{rostateduality}. Now we claim that for an abelian group $B$, finiteness of $B^*$ implies finiteness of $B$. Assuming this, we only need to show that the map ${\rm{H}}^1(k, U) \rightarrow \oplus_v {\rm{H}}^1(k_v, U)$ has finite cokernel, and this is \cite[Chap. IV, \S2.6, Prop.\,(b)]{oesterle}. 

It only remains to prove the claim that for an abelian group $B$, finiteness of $B^*$ implies finiteness of $B$. For this, it suffices to show that the map $B \rightarrow B^{**}$ is injective. That is, we need to show that for $0 \neq b \in B$, there exists $\phi \in B^*$ such that $\phi(b) \neq 0$. Clearly there exists a homomorphism $\phi': \langle b \rangle \rightarrow \Q/\Z$ such that $\phi'(b) \neq 0$, where $\langle b \rangle \subset B$ is the group generated by $b$. Since $\Q/\Z$ is divisible, it is an injective abelian group, hence we may extend $\phi'$ to a homomorphism $\phi: B \rightarrow \Q/\Z$, and then we have $\phi(b) \neq 0$, as required. This completes the proof of the Theorem.
\end{proof}

\begin{proof}[Proof of Theorem $\ref{extisfg}$]
By Lemma \ref{extexactsequence}, if the theorem holds for the two end terms in a short exact sequence, then it also holds for the middle term.

First assume that char$(k) = 0$, i.e., $k$ is a number field. Then Chevalley's Theorem reduces us to the cases in which $G$ is either a linear algebraic or an abelian variety. The former case is well--known (over arbitrary perfect fields). If $G = A$ is an abelian variety, then Ext$^1(A, \mathbf{G}_m)$ is the set of primitive line bundles on $A$, which equals Pic$^0(A) = A'(k)$, where $A'$ is the dual abelian variety \cite[\S13]{mumford}. The desired result therefore follows from the Mordell--Weil Theorem in this case.

Now suppose that char$(k) = p > 0$. Then by \cite[Thm.\,A.3.9]{cgp}, $G$ is an extension of a connected linear algebraic group by a semi-abelian variety. We may therefore assume that $G$ is either affine or an abelian variety. The affine case is handled by Theorem \ref{extisfinite}, and the abelian variety case once again follows from the Mordell--Weil Theorem.
\end{proof}

\section{Pathologies with forms of $\Ga$}
\label{examples}

In this section we discuss forms of $\Ga$ with a view toward constructing examples of various types of pathological behavior. (Proposition \ref{extfiniteformsofGa} is one indication that these groups may be a source of various pathologies.)

\begin{definition}
\label{genus}
Let $k$ be a field, $U$ a $k$-form of $\Ga$. Let $C$ be the regular compactification of $U$. That is, $C$ is the unique regular proper curve over $k$ equipped with an open embedding $U \hookrightarrow C$ with dense image. Then we define the genus of $U$ to be the arithmetic genus ${\rm{h}}^1(C, \mathcal{O}_C)$ of $C$.
\end{definition}

Since $C$ is proper, its Picard functor $\Pic_{C/k}$ is representable by a locally finite type $k$-group scheme. We will use this group scheme to study the Picard group of $U$.

\begin{proposition}
\label{jac=unipotent}
Let $U$ be a $k$-form of $\Ga$. Then the Jacobian ${\rm{Jac}}(C) := \Pic^0_{C/k}$ is a smooth connected commutative wound unipotent $k$-group of dimension equal to the genus of $U$.
\end{proposition}

\begin{proof}
The group ${\rm{Jac}}(C)$ is connected by definition, and clearly commutative. It is smooth of dimension equal to the arithmetic genus of $C$, as holds for any curve by cohomological considerations. The only thing that has to be checked, therefore, is that $Jac(C)$ is wound unipotent. This follows from \cite[Thm.\,4.4]{achet}.
\end{proof}

Before stating the next proposition, we recall the notion of woundness, which is the analogue for unipotent groups of anisotropicity for tori. A {\em $k$-wound} (or simply wound if $k$ is clear from context) unipotent group over a field $k$ is a smooth connected unipotent $k$-group scheme $U$ such that any map  of $k$-schemes $\A^1_k \rightarrow U$ from the affine line to $U$ is a constant map to some $u \in U(k)$. This is equivalent to saying that $U$ does not contain a copy of $\Ga$ \cite[Def.\,B.2.1, Cor.\,B.2.6]{cgp}. Woundness is insensitive to separable field extension. That is, if $U$ is a smooth connected unipotent $k$-group, and $K/k$ is a (not necessarily algebraic) separable field extension, then $U$ is $k$-wound if and only if $U_K$ is $K$-wound \cite[Prop.\,B.3.2]{cgp}. Clearly, smooth connected $k$-subgroups of wound unipotent $k$-groups are still wound. So are extensions of wound groups by other wound groups, as may be checked by using the formulation in terms of containing a copy of $\Ga$. Quotients of wound groups, however, need not be wound (even quotients by smooth connected subgroups). Over a perfect field, all smooth connected unipotent groups are split \cite[Thm.\,15.4(iii)]{borel}, hence no nontrivial wound unipotent groups exist over such fields. Over imperfect fields, however, there are many. Some examples are given in Example \ref{woundunipexamples}. Non-commutative examples are given in \cite[Ex.\,2.10]{conrad2}.

\begin{proposition}
\label{C-U}
Let $U$ be a form of $\Ga$ over a field $k$, and let $C$ be the regular compactification of $U$. Then $C - U$ consists of a single closed point $x$ of $C$ that becomes rational over some finite purely inseparable extension of $k$. The group $U$ is $k$-wound if and only if $x$ is not $k$-rational.
\end{proposition}

\begin{proof}
The first assertion of the proposition follows from \cite[Lemma 1.1(i)]{russell}. If $U$ is not $k$-wound, then it contains a copy of $\Ga$, hence is $k$-isomorphic to $\Ga$ since it is one-dimensional. Then $C \simeq \mathbf{P}^1_k$, and $C - U$ consists of a rational point. Conversely, assume that $x$ is $k$-rational, and we will show that $U$ is not $k$-wound. If $x \in C(k)$, then since $C$ is regular, it is smooth at $x$, hence smooth everywhere. But this implies that $\Jac(C)$ is an abelian variety. By Proposition \ref{jac=unipotent}, therefore, $\Jac(C) = 0$. This implies that $C \simeq \mathbf{P}^1_k$. Then $U = C - \{x\} \simeq \A^1_k$, hence $U$ is not wound.
\end{proof}

\begin{proposition}
\label{picjacseq}
Let $U$ be a form of $\Ga$ over the field $k$, and let $J$ denote the Jacobian of the regular compactification $C$ of $U$. Then we have an exact sequence
\[
0 \longrightarrow J(k) \longrightarrow \Pic(U) \xlongrightarrow{{\rm{deg}}} \Z/p^n\Z \longrightarrow 0,
\]
where $p^n$ is the degree of the unique point of $C - U$. Here, the first map is the restriction map $\Pic^0(C) \rightarrow \Pic(U)$, and the second is given by sending $\mathscr{L} \in \Pic(U)$ to the degree of any line bundle on $C$ that restricts to $\mathscr{L}$.
\end{proposition}

\begin{proof}
This exact sequence is \cite[(2.1.3)]{achet}.
\end{proof}

\begin{proposition}
\label{positivegenus}
A $k$-form $U$ of $\Ga$ has genus $0$ if and only if it is $k$-rational. If $p > 2$, then this holds if and only if $U \simeq \Ga$, while if $p = 2$, then this holds if and only if $U \simeq \Ga$ or $U$ is $k$-isomorphic to the subgroup of $\Ga^2$ given by the equation $Y^2 = X + aX^2$ for some $a \in k - k^2$.
\end{proposition}

\begin{proof}
The first assertion is a classical fact classifying regular proper curves of genus $0$ which contain a rational point. Further, the group $U$ is a so-called Russell group \cite[\S2.6]{kmt}, so the second assertion is \cite[Thm.\,6.9.2]{kmt}.
\end{proof}

\begin{lemma}
\label{multiplicationextends}
Let $k$ be a field, and let $U$ be a one-dimensional smooth connected $k$-group scheme. Let $C$ denote the regular compactification of $U$. Then the action $U \times U \rightarrow U$ of $U$ on itself by left-multiplication extends uniquely to an action $U \times C \rightarrow C$ of $U$ on $C$.
\end{lemma}

\begin{proof}
The scheme $U \times C$ is smooth over $C$, hence regular. In particular, $U \times C$ is reduced and $C$ is separated, so if the action extends, then it extends uniquely. Similarly, if the multiplication map $U \times U \rightarrow U$ extends at all to a map $U \times C \rightarrow C$, then that map is a group scheme action of $U$ on $C$.

To show that it extends, we may extend scalars to $k = k_s$ (over which $C$ remains regular, since this is preserved by \'etale base change) by using Galois descent and the uniqueness of the extension. We note that the regularity of $U \times C$ and the properness of $C$ together imply that the map $U \times U \rightarrow U \subset C$ extends to a rational map $\pi: U \times C \dashrightarrow C$ defined away from finitely many closed points of the 2-dimensional scheme $U \times C$. Since $k = k_s$ and $U$ is smooth, $U(k)$ is Zariski dense in $U$, so we may choose $u_0 \in U(k)$ such that $\pi$ is defined on $u_0 \times C$, and such that the loci on which the maps $\pi(u, c)$ and $\pi(u_0+u, c)$ are not defined are disjoint. Then the map $f(u, c) := \pi(-u_0, \pi(u_0+u, c))$ is defined everywhere that $\pi$ is not, and they agree on $U \times U$, hence on their entire locus of common definition. They therefore glue to give a map $U \times C \rightarrow C$ extending the left action $U \times U \rightarrow U$.
\end{proof}

\begin{proposition}
\label{extcontainsavariety}
Let $U$ be a $k$-form of $\mathbf{G}_a$ of positive genus. Then there is a smooth connected positive-dimensional closed $k$-subgroup scheme $G \subset \Jac(C)$ such that $G(k) \subset \Ext^1(U, \Gm)$ via the inclusion $G(k) \subset \Jac(C)(k) = \Pic^0(C) \hookrightarrow \Pic(U)$.
\end{proposition}

\begin{proof}
We adapt an idea of Totaro's from the proof of \cite[Lem.\,9.4]{totaro}. The action described by Lemma \ref{multiplicationextends} yields an action $U \times \Jac(C) \rightarrow \Jac(C)$. Via the inclusion $\Jac(C)(k) \hookrightarrow \Pic(U)$, this action is precisely the translation action of $U$ on its Picard group. For an element of $\Pic(U)$ to lie in $\Ext^1(U, \Gm)$ is the same as being fixed by this action. It therefore suffices to show that there is a smooth connected positive-dimensional closed $k$-subgroup scheme $G \subset \Jac(C)$ such that $U$ acts trivially on $G$.

Consider the unipotent group $N := \Jac(C) \rtimes U$, and consider the sequence of smooth connected $k$-subgroups $P_i \subset \Jac(C)$ defined inductively by $P_0 = \Jac(C)$, $P_{n+1} := [N, P_n]$. Since $N$ is unipotent, it is a nilpotent group, so $P_n = 0$ for some $n$. If $U$ has positive genus, and if we let $n$ be the first such index, then $n > 0$, and $U$ acts trivially on $P_{n-1}$.
\end{proof}

\begin{example}
\label{woundunipexamples}
Let us remark that Proposition \ref{extcontainsavariety} applies to a nonempty set of examples! 
In fact, there are many examples over any imperfect field $k$ of $k$-forms of $\mathbf{G}_a$ of positive genus. First consider the case $p > 2$, let $a \in k - k^p$, and let $U := \{ y^p = x + ax^p \} \subset \mathbf{G}_a^2$. Then the regular compactification of $U$ is the curve $C := \{ y^p = xz^{p-1} + ax^p \} \subset \mathbf{P}^2_k$, which has genus ${p-1 \choose 2} > 0$.

Now suppose that $p = 2$. Then we take $U := \{ y^{p^2} = x + ax^{p^2} \}$, where once again $a \in k- k^p$. Then the regular compactification of $U$ is $C := \{ y^{p^2} = xz^{p^2-1} + ax^{p^2} \}$, which has genus ${p^2 - 1 \choose 2} > 0$. This example also works for $p > 2$, but we wanted to give as simple an example as possible in each case.
\end{example}

Recall that a local function field is the completion of a global function field at some place. Equivalently, it is the field $\kappa((T))$ of Laurent series in a single variable over some finite field $\kappa$. As a consequence of Proposition \ref{extcontainsavariety}, we obtain the following result.

\begin{proposition}
\label{infiniteextseparablyclosed}
Let $k$ be an imperfect separably closed field or a local function field, and let $U$ be a $k$-form of $\Ga$. Then $\Ext^1(U, \Gm)$ is finite if and only if $U \simeq \Ga$, or if $p=2$ and $U \simeq \{Y^2 = X + aX^2\}$ for some $a \in k - k^2$.
\end{proposition}

\begin{proof}
By Proposition \ref{positivegenus}, the proposition is saying that $\Ext^1(U, \Gm)$ is finite if and only if $U$ has genus $0$. If $U$ has genus $0$, then $\Jac(C) = 0$, so $\Ext^1(U, \Gm)$ is finite by Proposition \ref{picjacseq}.

Next suppose that $U$ has positive genus, so $\Jac(C)$ is positive-dimensional. Choosing $G$ as in Proposition \ref{extcontainsavariety}, we see that it suffices to show that any smooth $k$-group scheme $G$ of positive dimension has infinitely many rational points. This is true when $k = k_s$, since then $G(k)$ is Zariski dense in $G$, and when $k$ is a local function field, since then $G(k)$ is a positive-dimensional Lie group over $k$.
\end{proof}

\begin{remark}
Proposition \ref{infiniteextseparablyclosed} shows that the conclusion of Theorem \ref{extisfinite} fails over every imperfect separably closed field and every local function field.
\end{remark}

\begin{lemma}
\label{notlinequiv}
Let $k$ be a field, and let $U$ be a $k$-form of $\Ga$ of positive genus. Then for any distinct points $x, y \in U(k_s)$, the divisors $[x]$ and $[y]$ are not linearly equivalent.
\end{lemma}

\begin{proof}
We may assume that $k = k_s$. Suppose to the contrary that there is a rational function $f$ on $U$ with divisor $[x] - [y]$. This yields a map $f: U \rightarrow \mathbf{P}^1_k$ such that $f^{-1}(0) = [x]$ and $f^{-1}(\infty) = [y]$ scheme-theoretically. Since $C$ is regular, this extends to a map $\overline{f}: C \rightarrow \mathbf{P}^1_k$, which still has divisor $[x] - [y]$, since ${\rm{deg}}({\rm{div}}(\overline{f})) = 0$, and $C - U$ consists of a single closed point. Then $\overline{f}$ has degree one, hence is an isomorphism, which violates our assumption that $U$ has positive genus.
\end{proof}

\begin{proposition}
\label{pic=/=ext}
Let $U$ be a one-dimensional smooth connected unipotent group over a field $k$ $($i.e., a $k$-form of $\Ga$$)$. Then $\Pic(U) = \Ext^1(U, \Gm)$ if and only if $U \simeq \Ga$, or if $p = 2$ and $U$ is $k$-isomorphic to a group of the form $\{Y^2 = X + aX^2\} \subset \Ga \times \Ga$ for some $a \in k - k^2$.
\end{proposition}

\begin{proof}
By Proposition \ref{positivegenus}, the proposition is the same as saying that $\Pic(U) = \Ext^1(U, \Gm)$ if and only if $U$ has genus $0$. First suppose that $U$ has genus $0$. We want to prove that $\Pic(U) = \Ext^1(U, \Gm)$. We claim that the map $\Pic(U \times U) \rightarrow \Pic((U \times U)_{k_s})$ is injective. Indeed, letting $W := U \times U$, then since $W$ is unipotent, $\widehat{W}(k_s) = 0$, so the claim follows from \cite[Lem.\,6.3(ii) and 6.5(iii)]{sansuc}.

It therefore suffices to check the assertion for genus $0$ groups after extending scalars to $k_s$, since this preserves the genus of $U$ (since \'etale base change preserves the regularity of $C$). But then $J = 0$ by Proposition \ref{jac=unipotent}, so $\Pic(U)$ is finite by Proposition \ref{picjacseq}. By Theorem \ref{picismultiplicativeforgroupswithfinitepic}(ii), therefore (which was proved in \S \ref{sectionfinitepic}), $\Pic(U) = \Ext^1(U, \Gm)$.

Conversely, suppose that $U$ has positive genus. We claim that the line bundle $\mathscr{L} := \mathcal{O}([0])$ corresponding to the divisor $0 \in U(k)$ is not in $\Ext^1(U, \Gm)$. This may be checked after extending scalars to $k_s$. Since $U_{k_s}$ has the same genus as $U$, therefore, we may assume that $k = k_s$. In particular, $U(k) \neq \{0\}$. Lemma \ref{notlinequiv} then implies that for any $0 \neq u \in U(k)$, $t_u^*\mathscr{L} \not \simeq \mathscr{L}$, where $t_u: U \rightarrow U$ is translation by $u$. Therefore, $\mathscr{L}$ is not translation-invariant, hence not in $\Ext^1(U, \Gm)$.
\end{proof}

We now turn to the construction of various pathologies. 

\begin{example}
Let $U$ be a $k$-form of $\Ga$ such that $\Pic(U)$ is infinite. Proposition \ref{infiniteextseparablyclosed} provides such examples over every local function field and every imperfect separably closed field. We will use $U$ to construct a commutative pseudo-reductive $k$-group with infinite Picard group.

Because it has infinite Picard group, the group $U$ is $k$-wound. By \cite[Cor.\,9.5]{totaro}, there is a (smooth connected) commutative pseudo-reductive $k$-group $G$ such that we have an exact sequence
\[
1 \longrightarrow T \longrightarrow G \longrightarrow U \longrightarrow 1
\]
with $T$ a $k$-torus. We claim that ${\rm{Pic}}(G)$ is infinite. Indeed, by \cite[Cor.\,6.11]{sansuc} (the statement of that corollary assumes that the groups are all reductive, but the proof only uses the reductivity of the kernel $G'$), we have an exact sequence
\[
\widehat{G}(k) \longrightarrow \widehat{T}(k) \longrightarrow {\rm{Pic}}(U) \longrightarrow {\rm{Pic}}(G).
\]
Since $\mbox{coker}\left(\widehat{G}(k) \rightarrow \widehat{T}(k)\right)$ is finite (Lemma \ref{finitecokernelchars}) and ${\rm{Pic}}(U)$ is infinite, ${\rm{Pic}}(G)$ must be infinite as well.
\end{example}

\begin{example}
\label{zariskidenseneeded}
Let us construct an example to show that the hypothesis that $G(k)$ is Zariski dense in $G$ is necessary in Theorem \ref{picismultiplicativeforgroupswithfinitepic}(i), and therefore the corresponding hypothesis for $H$ is necessary in Theorem \ref{Pic(GXH)}(i).

By Propositions \ref{pic=/=ext} and \ref{picjacseq}, to construct a $k$-form $U$ of $\Ga$ such that $\Pic(U)$ is finite but not equal to $\Ext^1(U, \Gm)$ is the same as giving $U$ of positive genus such that $J(k)$ is finite.

If $k$ is an arbitrary imperfect field, then there may not exist such $U$. For example, if $k$ is an imperfect separably closed field, then $J(k)$ must be infinite if $J \neq 0$. The same holds if $k$ is a local function field, because $J(k)$ is a positive-dimensional Lie group over $k$.

But let $k$ be a global function field, and $U$ a wound $k$-form of $\Ga$ such that the genus $g$ of $U$ satisfies $0 < g < p-1$. Then we claim that $J(k)$ is finite, and so $\Pic(U)$ is finite but not equal to $\Ext^1(U, \Gm)$. Indeed, $J$ is a wound unipotent group of dimension $g < p-1$ by Proposition \ref{jac=unipotent}, so by \cite[Ch.\,VI, \S3.1, Thm.]{oesterle}, $J(k)$ is finite. Here is an example when ${\rm{char}}(k) = 3$. Consider the group $U := \{Y^3 = X + aX^3\} \subset \Ga \times \Ga$ with $a \in k - k^3$. Then one may check as before that the projectivization $\{Y^3 = XZ^2 + aX^3\} \subset \mathbf{P}^2_k$ is regular, hence yields the regular compactification of $U$. Further, one may compute its genus to be $1 < 3 - 1$.
\end{example}

Let us give some applications to the cohomology of local and global fields; in particular, we will show that the groups ${\rm{H}}^1(k, U)$ may exhibit certain pathologies over such fields. Recall that if $k$ is a local field of characteristic $0$, then for any affine commutative $k$-group scheme $G$ of finite type, the group ${\rm{H}}^1(k, G)$ is finite.
This is completely false over local function fields, e.g., for $G = \alpha_p$, $\mu_p$, or $\Z/p\Z$. But one would like to give smooth connected examples. The following proposition provides a plethora of such groups.

\begin{proposition}
\label{H^1infinite}
Let $k$ be a local field of positive characteristic, and let $G$ be a connected commutative linear algebraic $k$-group. Then ${\rm{H}}^1(k, G)$ is finite if and only if $\Ext^1(G, \Gm)$ is finite. In particular, if $U$ is a $k$-form of $\Ga$, then ${\rm{H}}^1(k, U)$ is finite if and only if $U \simeq \Ga$, or if $p = 2$ and $U \simeq \{Y^2 = X + aX^2\}$ for some $a \in k- k^2$. If $U$ is $k$-wound, then ${\rm{H}}^1(k, U) \neq 0$.
\end{proposition}

\begin{proof}
The second assertion follows from the first and Proposition \ref{infiniteextseparablyclosed}.
To prove 
the first and last assertion, we use the fact that local Tate duality implies that the groups ${\rm{H}}^1(k, G)$ 
and ${\rm{H}}^1(k, \widehat{G})$ are Pontryagin dual \cite[Thm.\,1.2.4]{rostateduality}. We further have an isomorphism ${\rm{H}}^1(k, 
\widehat{G}) \simeq \Ext^1(G, \Gm)$ \cite[Cor.\,2.3.4]{rostateduality}. (Recall that the functorial Ext and our Ext agree in this setting, by Proposition \ref{derivedextagrees}.) The first assertion of the lemma follows immediately. 

To prove the last assertion, we need to show that when $p = 2$ and $U \simeq \{Y^2 = X + aX^2\}$ for some $a \in k- k^2$, then $\Ext^1(U, \Gm) \neq 0$. By Proposition \ref{pic=/=ext}, it suffices to show that $\Pic(U) \neq 0$. This follows from the woundness of $U$ and Propositions \ref{picjacseq} and \ref{C-U}.
\end{proof}

\begin{remark}
More generally, N.\,Duy T$\hat{{\rm{a}}}$n determined precisely which unipotent groups over local fields of positive characteristic have finite ${\rm{H}}^1$ \cite[Thm.\,10]{duy}.
\end{remark}

Let $G$ be a group scheme over a global field $k$, and let $S$ be a finite set of places of $k$. Then we define the pointed set $\Sha_S(G)$ (often also denoted $\Sha^1_S(G)$) in the usual manner:
\[
\Sha_S(G) := \ker \left( {\rm{H}}^1(k, G) \longrightarrow \prod_{v \notin S} {\rm{H}}^1(k_v, G) \right).
\]
If $k$ is a number field, then this set is finite for any affine $k$-group scheme of finite type \cite[Thm.\,7.1]{borelserre}. This finiteness also holds in the function field setting if $S = \emptyset$ \cite[Thm.\,1.3.3(i)]{conrad}. It is not generally true, however, for nonempty $S$ (contrary to the original version of \cite[Thm.\,1.3.3(i)]{conrad}), as we show in Corollary \ref{formGainfinitesha} below.

\begin{remark}
\label{conraderrata}
The mistake in the proof of \cite[Thm.\,1.3.3(i)]{conrad} creeps in when invoking \cite[Chap.\,IV, \S2.6, Prop.\,(a)]{oesterle} at the very end of \S6.3 and in the first paragraph of \S6.4. In fact, Oesterl\'e's result only asserts the finiteness of $\Sha^1(G)$ (for $G$ a smooth connected affine solvable group over a global field); that is, it only treats the case $S = \emptyset$. For this reason, all of the arguments in \cite{conrad}, and particularly \cite[Thm.\,1.3.3(i)]{conrad}, are valid when $S = \emptyset$. But as we are about to show, things go haywire when $S \neq \emptyset$.
\end{remark}

\begin{proposition}
\label{infiniteSha}
Let $G$ be a connected commutative linear algebraic group over a global function field $k$, and let $S$ be a finite set of places of $k$. Then the group $\Sha^1_S(G)$ is finite if and only if $\Ext^1_{k_v}(G, \Gm)$ is finite for all $v \in S$.
\end{proposition}

\begin{proof}
Let $\A^S$ denote the ring of $S$-adeles, i.e., the subset of $\prod_{v \notin S} k_v$ defined by the usual restricted product condition. Then \cite[Prop.\,5.2.2]{rostateduality} implies that 
$$\Sha^1_S(G) = \ker \left( {\rm{H}}^1(k, G) \longrightarrow {\rm{H}}^1(\A^S, G) \right).$$
Global Tate duality \cite[Thm.\,1.2.8]{rostateduality} furnishes an exact sequence
\[
\Sha^1(G) \longrightarrow \Sha^1_S(G) \longrightarrow \prod_{v \in S} {\rm{H}}^1(k_v, G) \longrightarrow {\rm{H}}^1(k, \widehat{G})^*,
\]
where, as before, $\widehat{G} := \calHom(G, \Gm)$ denotes the fppf $\Gm$-dual sheaf of $G$, and for an abelian group $B$, $B^* := {\rm{Hom}}(B, \Q/\Z)$. The group $\Sha^1(G)$ is finite \cite[Chap.\,IV, \S2.6, Prop.\,(a)]{oesterle}, and the group ${\rm{H}}^1(k, \widehat{G})$ is isomorphic to $\Ext^1(G, \Gm)$ \cite[Cor.\,2.3.4]{rostateduality}, which is also finite by Theorem \ref{extisfinite}. It follows that $\Sha^1_S(G)$ is finite if and only if ${\rm{H}}^1(k_v, G)$ is finite for all $v \in S$. By Proposition \ref{H^1infinite}, these groups are finite if and only if the groups $\Ext^1_{k_v}(G, \Gm)$ are finite for all $v \in S$.
\end{proof}

The following result shows that, in contrast to the number field setting, the set $\Sha_S(G)$ can be infinite even for smooth connected affine groups over global function fields if $S \neq \emptyset$. (This cannot happen if $S = \emptyset$; see Remark \ref{conraderrata}.)

\begin{corollary}
\label{formGainfinitesha}
Let $k$ be a global function field, $U$ a nontrivial $k$-form of $\Ga$. If char$(k) = 2$, then further assume that $U$ is not isomorphic to a conic in the affine plane $($equivalently, by Proposition $\ref{positivegenus}$, $U$ has positive genus$)$. Let $S$ be a nonempty set of places of $k$. Then $\Sha^1_S(U)$ is infinite.
\end{corollary}

\begin{proof}
We may assume that $S$ is finite. Then this is an immediate consequence of Propositions \ref{infiniteSha} and \ref{infiniteextseparablyclosed}, together with the fact that the genus of $U$ is preserved upon passage from $k$ to $k_v$, since $k_v$ is a (non-algebraic) separable extension of $k$. (This means that $k_v$ is a filtered direct limit of smooth $k$-algebras; for other equivalent definitions of separability, see \cite[\S 27, Lemma 3]{matsumura}.)
\end{proof}

\appendix

\section{Characters}

For a $k$-group scheme $G$, $\widehat{G}(k)$ denotes the group Hom$(G, \Gm)$ of characters of $G$ defined over $k$. The following result is \cite[Prop.\,A.1.4]{oesterle}. We state it here for convenience.

\begin{lemma}
\label{finitecokernelchars}
Let $G$ be a connected linear algebraic $k$-group, and $H \trianglelefteq G$ a normal $k$-subgroup scheme. Then the restriction map $\widehat{G}(k) \rightarrow \widehat{H}(k)$ has finite cokernel.
\end{lemma}

\begin{lemma}
\label{finitelygenerated}
If $G$ is a $k$-group scheme of finite type, then $\widehat{G}(k)$ is finitely generated.
\end{lemma}

\begin{proof}
We may assume that $k = \overline{k}$. Given a short exact sequence
\[
1 \longrightarrow G' \longrightarrow G \longrightarrow G'' \longrightarrow 1
\]
of finite type $k$-group schemes, we get an exact sequence
\[
0 \longrightarrow \widehat{G''}(k) \longrightarrow \widehat{G}(k) \longrightarrow \widehat{G'}(k),
\]
so if the Lemma holds for $G'$ and $G''$, then it holds for $G$. The lemma is clear for finite group schemes. By \cite[VII$_{\rm{A}}$, Prop.\,8.3]{sga3}, there is a normal infinitesimal $k$-subgroup scheme $I \trianglelefteq G$ such that $G/I$ is smooth. We may therefore assume that $G$ is smooth, and replacing $G$ with $G^0$, we may assume that $G$ is smooth and connected. By Chevalley's Theorem, we may further assume that $G$ is either an abelian variety or a linear algebraic group. In the former case $\widehat{G}(k) = 0$ since $G$ is proper and $\Gm$ is affine. In the latter case, we may replace $G$ with $G/\mathscr{D}G$ (since any character of $G$ factors through this quotient) and thereby assume that $G$ is commutative. Letting $T \subset G$ be the maximal torus, $G/T$ is unipotent and therefore has no nontrivial characters. We may therefore assume that $G = T$ is a torus, and since $k = \overline{k}$, we are reduced to the easy case $G = \Gm$.
\end{proof}

Recall that $k_{{\rm{per}}}$ denotes a perfect closure of $k$.

\begin{lemma}
\label{finitequotientchars}
Let $G$ be a $k$-group scheme of finite type. Then $\widehat{G}(k_{per})/\widehat{G}(k)$ is finite.
\end{lemma}

\begin{proof}
We may assume that ${\rm{char}}(k) = p > 0$. The quotient group is finitely generated by Lemma \ref{finitelygenerated}, so in order to show that it is finite, it suffices to show that it is torsion. In fact, given $\chi \in \widehat{G}(k_{{\rm{per}}})$, we have $\chi^{p^n} \in \widehat{G}(k)$ for some $n > 0$.
\end{proof}

Lemma \ref{finitequotientchars} is not true if we replace the extension $k_{{\rm{per}}}/k$ with even arbitrary finite extensions. For example, let $T$ be a $k$-torus with no nontrivial characters over $k$. For example, any nontrivial form of $\Gm$ satisfies this, since any isogeny of tori has an isogeny splitting. This may also be seen by noting that the anti-equivalence between tori and Galois lattices (via $T \mapsto \widehat{T}(k_s)$) implies that any such $T$ corresponds to a nontrivial action of ${\rm{Gal}}(k_s/k)$ on $\Z$, and any such action has no nonzero fixed elements. At any rate, if $L/k$ is a finite separable extension splitting $T$, then $\widehat{T}(L)/\widehat{T}(k)$ is infinite.

\noindent \address
\vspace{.3 in}

\noindent \email

\end{document}